\documentclass[reqno]{amsart}
\usepackage{hyperref}
\usepackage{amssymb, enumerate}

\begin{document}
\title[\hfil Pohozaev type ineq.]
{Pohozaev-type inequalities and nonexistence results for non $C^2$ solutions of $p(x)$-laplacian equations.}

\author[Gabriel L\'opez G ]
{Gabriel L\'opez}

\address{
Gabriel L\'opez G. \newline
Universidad Aut\'onoma Metropolitana , M\'exico D.F, M\'exico}
\email{gabl@xanum.uam.mx}
\

\subjclass[2000]{35D05, 35J60, 58E05}
\keywords{Pohozaev-type inequality, $p(x)$-Laplace
operator, variable exponent Sobolev spaces.}

\begin{abstract}
In this paper a Pohozaev type inequality is stated for variable exponent Sobolev spaces in order to prove
non existence of nontrivial weak solutions for a Dirichlet problem with non-standard growth. The obtained results generalize a previous work of M. \^{O}tani.
\end{abstract}

\maketitle
\numberwithin{equation}{section}
\newtheorem{defi}{Definition}[section]
\newtheorem{theorem}[defi]{Theorem}
\newtheorem{lemma}[defi]{Lemma}
\newtheorem{coro}[defi]{Corollary}
\newtheorem{rem}{Remark}
\newtheorem{prop}[defi]{Proposition}

\section{Introduction}
Let $\Omega$ be a bounded domain in $\mathbb{R}^N$ with smooth boundary $\partial \Omega.$ The domain $\Omega$ is said  to be {\it star shaped} (respectively {\it strictly star shaped}) if $(x\cdot\nu(x))\geqslant 0$  (respectively if $(x\cdot \nu(x))\geqslant \rho>0$) holds for all $x\in\partial \Omega$ with a suitable choice of the origin, where
$\nu(x)=(\nu_1(x),\dots,\nu_N(x))$ denotes the outward normal unit vector at $x\in\partial \Omega.$
Consider the problem
\begin{eqnarray}
\label{DI}
\begin{cases}
-\Delta_{p(x)} u = f(u) ,\qquad x\in \Omega\\
u(x)=0,\qquad x\in\partial\Omega.
\end{cases}
\end{eqnarray}
In \cite{di} in order to obtain some non existence results for Problem (\ref{DI}) with $\Omega$ star shaped  some Pohozaev type identities are stated and applied to the case in which $f$ does not depend of $p(x)$ and $u\in C^2(\Omega).$ Nevertheless,
it is known \cite{ho} that for $f(u)=|u|^{q-2}u,$ $1<q<\infty,$  $2<p<\infty,$ and $p,q$ constants, non\-trivial solutions of (\ref{DI}) does not belong to $C^2(\Omega)\cap C(\overline{\Omega}).$ The arguments in \cite[Proposition 1.1]{ho} are easily extended to the variable exponent case, so that in general, results in \cite{di} can not be applied when $\nabla u(x)=0,$ not even for solutions in $W^{2,p(x)}(\Omega)\cap W^{1,p(x)}(\overline{\Omega}).$ In this way,  solutions of the problem,
\begin{eqnarray}
\label{E}
(E)
\begin{cases}
-\Delta_{p(x)} u = |u|^{q(x)-2}u ,\qquad x\in \Omega\\
u(x)=0,\qquad x\in\partial\Omega,
\end{cases}
\end{eqnarray}
where $\Delta_{p(x)}u=\text{div}(|\nabla u|^{p(x)-2}\nabla u),$ in general do not belong to $C^2(\Omega).$

 Existence of solutions for problem $(E)$ is studied in \cite{fzh} and \cite{moss}. The authors in \cite{moss}   prove existence  for the case in which the embbeding from $W_0^{1,p(\cdot)}(\Omega)$ to $L^{q(\cdot)}(\Omega)$ is compact and moreover,  they prove existence even for the case in which the embbeding from $W_0^{1,p(\cdot)}(\Omega)$ to $L^{q(\cdot)}(\Omega)$ is not compact provided that certain functional inequality holds true.

This paper is organized as follows. In section \ref{VES} some necessary background in Variable Exponent Sobolev Spaces is provided including some required Compact Embedding results.    In section \ref{ptis}, Theorem \ref{l4.2} we state and prove a Pohozaev-type inequality. In Section \ref{NES}, as a consequence of the Pohozaev type inequality, we prove some nonexistence results of nontrivial weak solutions of problem (\ref{E}).

\section{Variable exponent setting}
\label{VES}
We recall some definitions and  basic properties of the variable exponent
Lebesgue-Sobolev spaces $L^{p(\cdot)}(\Omega)$ and
$W_0^{1,p(\cdot)}(\Omega)$, where $\Omega$ is a bounded domain in
$\mathbb{R}^N$.

 For any
$p\in \mathcal{C}(\overline\Omega)$ we define
$$
p^+=\sup_{x\in\Omega}p(x)\quad\mbox{and}\quad
p^-=\inf_{x\in\Omega}p(x).
$$
 The variable
exponent Lebesgue space for
 measurable real-valued  functions is defined as the set
\begin{align*}
L^{p(\cdot)}(\Omega)=\left\{u: \int_\Omega \left|u(x) \right|^{p(x)}\,dx<\infty\right\},
\end{align*}
endowed with the  {\it Luxemburg norm}
$$
\|u\|_{p(\cdot)}=\inf\left\{\mu>0;\;\int_\Omega
\left|\frac{u(x)}{\mu}\right|^{p(x)}\,dx\leq 1\right\},
$$
which is a separable and reflexive Banach space if $1<p^-\leqslant p^+<\infty.$ For basic
properties of the variable exponent Lebesgue spaces we refer to
 \cite{dhhr}, \cite{ko}.


Let $L^{p'(\cdot)}(\Omega)$ be the conjugate space of
$L^{p(\cdot)}(\Omega)$, obtained by conjugating the exponent
pointwise that is,  $1/p(x)+1/p'(x)=1$, \cite[Corollary~2.7]{ko}.
For any $u\in L^{p(\cdot)}(\Omega)$ and $v\in L^{p'(\cdot)}(\Omega)$
the following H\"older type inequality is valid

\begin{equation}\label{Hol}
\left|\int_\Omega uv\,dx\right|\leq\left(\frac{1}{p^-}+
\frac{1}{{p'}^-}\right)\|u\|_{p(\cdot)}\|v\|_{p'(\cdot)}.
\end{equation}

An important role in manipulating the generalized Lebesgue-Sobolev
spaces is played by the {\it $p(\cdot)$-modular} of the
$L^{p(\cdot)}(\Omega)$ space, which is the mapping
 $\rho_{p(\cdot)}:L^{p(\cdot)}(\Omega)\to\mathbb{R}$ defined by
$$
\rho_{p(\cdot)}(u)=\int_\Omega|u|^{p(x)}\,dx.
$$
If $(u_n)$, $u\in L^{p(\cdot)}(\Omega)$ then the following relations
hold
\begin{gather}\label{L40}
\|u\|_{p(\cdot)}<1\;(=1;\,>1)\;\Leftrightarrow\;\rho_{p(\cdot)}(u)
<1\;(=1;\,>1)
\\ \label{L4}
\|u\|_{p(\cdot)}>1 \;\Rightarrow\;
\|u\|_{p(\cdot)}^{p^-}\leq\rho_{p(\cdot)}(u) \leq \|u\|_{p(\cdot)}^{p^+}
\\ \label{L5}
\|u\|_{p(\cdot)}<1 \;\Rightarrow\; \|u\|_{p(\cdot)}^{p^+}\leq
\rho_{p(\cdot)}(u)\leq \|u\|_{p(\cdot)}^{p^-}
\\ \label{L6}
\|u_n-u\|_{p(\cdot)}\to 0\;\Leftrightarrow\;\rho_{p(\cdot)} (u_n-u)\to
0,
\end{gather}
since $p^+<\infty$. For a proof of these facts see \cite{ko}. 

The set $W_0^{1,p(x)}(\Omega)$ is defined as the closure of
$C_0^{\infty}(\Omega)$ under the norm
\[
\| u\|_{p(x)}=\|\nabla u\|_{p(x)}.
\]
The space $(W_0^{1,p(x)}(\Omega),\| \cdot \|_{p(x)})$ is a separable
and reflexive Banach space if $1<p^-\leqslant p^+<\infty.$ We note that if $q\in
C_+(\overline{\Omega})$ and $q(x)<p^*(x)$ for all $x\in
\overline{\Omega}$ then the embedding
$W_0^{1,p(x)}(\Omega)\hookrightarrow L^{q(x)}(\Omega)$ is  continuous, where $p^*(x)=Np(x)/(N-p(x))$ if $p(x)<N$ or
$p^*(x)=+\infty$ if $p(x)\geq N$ \cite[Theorem 3.9 and 3.3]{ko} (see
also \cite [Theorem 1.3 and 1.1]{fa}).

The bounded variable exponent $p$ is said to be Log-H\"older continuous
if there is a constant $C>0$ such that

\begin{equation}\label{logh}
  |p(x)-p(y)|\leqslant \frac{C}{-\log (|x-y|)}
\end{equation}
for all  $x,y \in \mathbb{R}^{N},$ such that $|x-y|\le\frac{1}{2}.$
A bounded exponent $p$ is Log-H\"older continuous in $\Omega$ if and
only if there exists a con\-stant $C>0$ such that

\begin{equation}\label{logh1}
|B|^{p^{-}_{B}-p^{+}_{B}}\le C
\end{equation}
for every ball $B\subset\Omega$ \cite[ Lemma 4.1.6, page 101]{dhhr}.
Under the Log-H\"older condition smooth functions are dense in varia\-ble   exponent Sobolev space \cite[Proposition 11.2.3, page 346]{dhhr}.

Finally, Compact Embedding results, as many other facts, are a very delicate and interesting issue in variable exponent spaces. For instance in \cite[prop 3.1]{moss} is shown that for certain exponents with $p^*(x)> q(x)>p^*(x)-\epsilon $ (in our notation) with $x$ in some subset of $\Omega$ the embedding from $W_0^{1,p(\cdot)}(\Omega)$ to $L^{q(\cdot)}(\Omega)$ is not compact. On the other hand, surprisingly, if $q(x)=p^*(x)$ at some point, it is known that the embedding is compact in $\mathbb{R}^N$ see \cite[Thm. 8.4.6]{dhhr} and references therein. In this paper we will use Proposition 3.3 of \cite{moss} which in our notation can be stated as the following Proposition.

\begin{prop}
\label{rkt}
  Let $p(\cdot)$ satisfying the log-H\"older condition on the open and bounded set $\Omega\subset \mathbb{R}^N.$ Suppose that $\partial\Omega\in C^1$ or $\Omega$ satisfies the cone condition, and  $p^+<N.$ Let $q(\cdot)$ be a variable exponent on $\Omega$ such that $1\leqslant q^-$ and
 \begin{equation}
 \label{rkc}
 \text{{\rm ess}}\inf_{x\in \Omega}\left(p^*(x)-q(x)\right)>0.
 \end{equation}
 Then $W_0^{1,p(\cdot)}(\Omega)\hookrightarrow\hookrightarrow L^{q(\cdot)}(\Omega),$ i. e. $W_0^{1,p(\cdot)}(\Omega)$ is compactly embedded in $L^{q(\cdot)}(\Omega).$
\end{prop}

In the next section we will require also the following Lemma.

\begin{lemma}\label{OTNOT}
Let $1<p(x)<q^-<q(x)<q^+<\infty$ $a.e.$ in $\Omega.$ Assume that $\|u_n\|_r<C$ for $1\leqslant r<\infty$ and $u_n\to u$ as $n\to \infty$ in $L^{p(\cdot)}(\Omega).$ Then $u_n\to u$ as $n\to \infty$ in $L^{q(\cdot)}(\Omega).$
\end{lemma}

\begin{proof}
Given (\ref{L40}) to (\ref{L6}) it is enough to show that $\rho_{q(\cdot)}(u_n-u)\to 0$ as $n\to\infty.$
For some $\theta\in (0,1)$ satisfying $1/q^-=\theta/p^-+(1-\theta)/q^+$ we have
\begin{multline}
\rho_{q(\cdot)}(u_n-u)=\int_\Omega |u_n-u|^{q(x)}dx\leqslant \int_\Omega |u_n-u|^{q^-}dx\\
\leqslant \left(\int_\Omega |u_n-u|^{p^-}dx\right)^{\theta q^-/p^-}\left(\int_\Omega |u_n-u|^{q^+}dx\right)^{(1-\theta) q^-/q^+}\\
\leqslant C\left(\int_\Omega |u_n-u|^{p^-}dx\right)^{\theta q^-/p^-}\to 0\text{ as }n\to\infty,
\end{multline}
given Thm. 2.11 in \cite{A}, and since $u_n\to u$ in $L^{p^-}(\Omega).$
\end{proof}



\section{Pohozaev-type inequalitiy}
\label{ptis}

In this section we state  a Pohozaev-type inequality for weak solutions $u$  belonging to the class $\mathcal{P}$ defined as
\begin{equation}
\label{P}
\mathcal{P}=\left\{u\in \left(W_0^{1,p(\cdot)}\cap L^{q(\cdot)}\right)(\Omega):x_i|u|^{q(x)-2}u\in L^{p'(\cdot)}(\Omega),\; i=1,2,\dots,N\right\}
\end{equation}
where $\;p'(x)=p(x)/(p(x)-1)$ and $p^+<N.$
To this aim, we employ the techniques introduced by  Hashimoto and $Ô$tani in \cite{ho}, \cite{h}, \cite{o1},  but within the framework of variable exponent spaces, which, as the reader may notice,  require much more careful estimations than those in the constant case.

Let $g_n(\cdot)\in C^1(\mathbb{R})$ be the cutoff functions such that $0\leqslant g'_n(s)\leqslant 1, \; s\in \mathbb{R}$ and
\begin{eqnarray}
\label{co}
g_n(s)=\begin{cases}
s\qquad |s|\leqslant n,\\
(n+1)\mbox{sign}\,s\qquad |s|\geqslant n+1.
\end{cases}
\end{eqnarray}
Let $u$ be a weak solution of (\ref{E}) 
and set $u_n=g_n(u)$ then $|u_n|^{r-2}u_n\in \left(W_0^{1,p(\cdot)}\cap L^\infty\right)(\Omega)$ for $r\in [2,\infty).$ Consider now the approximate problem

\begin{eqnarray}
\label{En}
(E)_n\quad\begin{cases}
|w_n|^{q(x)-2}w_n-\Delta_{p(x)}w_n=2|u_n|^{q(x)-2}u_n,\qquad\mbox{ in } \Omega,\\
w_n=0\qquad\mbox{ on }\partial \Omega.
\end{cases}
\end{eqnarray}
Since $u_n\in L^{\infty}(\Omega),$ there exists a sequence $\{v_n^\varepsilon\}\subset C_0^\infty(\Omega)$ satisfying
\begin{eqnarray}
\label{3o} \|v_n^\varepsilon\|_{L^\infty(\Omega)} \leqslant  C_o,\qquad \qquad\mbox{ for all } \varepsilon\in (0,1),\\
\label{4o} v_n^\varepsilon  \to  2|u_n|^{q(x)-2}u_n,\mbox{ strongly in }L^{r(\cdot)}(\Omega)\mbox{ as }\varepsilon\to 0,\mbox{ for all }r\in[1,\infty).
\end{eqnarray}
In turn, we require another approximate equation for $(E)_n$ given by
\begin{eqnarray}
\label{Ene}
(E)_n^\varepsilon\quad
\begin{cases}
|w_n^\varepsilon|^{q(x)-2}w_n^\varepsilon +A_\varepsilon w_n^\varepsilon=v_n^\varepsilon\qquad\mbox{ in }\Omega\\
w_n^\varepsilon=0\qquad\mbox{ on }\partial\Omega,
\end{cases}
\end{eqnarray}
where $A_\varepsilon u(x)=-\mbox{div}\left\{(|\nabla u(x)|^2+\varepsilon)^{(p(x)-2)/2}\nabla u(x) \right\}$ and $\varepsilon>0.$
It is possible to show that (\ref{En}) and (\ref{Ene}) have unique solutions and  that (\ref{Ene}) and (\ref{En}) provide good approximations respectively for (\ref{En}) and (\ref{E})  according to

\begin{lemma}
\label{3.1}
 Let $p(\cdot)$ satisfying the log-H\"older condition on the open and bounded set $\Omega\subset \mathbb{R}^N.$ Suppose that $\partial\Omega\in C^1$ or $\Omega$ satisfies the cone condition, and  $p^+<N.$
Then the following statements hold true:
\begin{enumerate}
\item[(i)] For each $\varepsilon\in (0,1)$ and $n\in\mathbb{N},$ there exists a unique solution $w_n^\varepsilon\in C^2(\overline{\Omega})$ of (\ref{Ene}).

\item[(ii)] For each $n\in \mathbb{N} $ there exists  a unique solution $w_n\in C^{1,\alpha}(\overline{\Omega})\cap W_0^{1,p(x)}(\Omega)$ of (\ref{En}).

\item[(iii)] $w_n^\varepsilon$ converges to $w_n$ as $\varepsilon\to 0$ in the following sense:
\begin{equation}
\label{6o}
\int_{\Omega}|\nabla w_n^\varepsilon|^{p(x)} dx\to\int_{\Omega}|\nabla w_n|^{p(x)} dx\quad\mbox{ as }\quad\varepsilon\to 0,
\end{equation}
\begin{equation}
\label{7o}
w_n^\varepsilon\to w_n\quad\mbox{ strongly in }L^r(x)(\Omega),
\end{equation}
for $r(\cdot)$ such that 
$1<r^-<r(x)<r^+$ $a.e.$ in $\Omega$
and $p^+<N.$
\item[(iv)] $w_n$ converges to $u$ as $n\to\infty$ in the following sense:
\begin{equation}
\label{8o}\int_{\Omega} |\nabla w_n|^{p(x)}dx\to\int_{\Omega} |\nabla u|^{p(x)}dx\mbox{ as }n\to \infty
\end{equation}
\begin{equation}
\label{9o}
\int_\Omega|w_n|^{q(x)}dx\to \int_\Omega |u|^{q(x)}dx,\quad\mbox{ as }\quad n\to\infty,
\end{equation}

\end{enumerate}
\end{lemma}

\begin{proof} (i)
Since $u_n\in L^{\infty}(\Omega),$ there exists a sequence $\{v_n^\varepsilon\}\subset C_0^\infty(\Omega)$ satisfying
\begin{eqnarray}
\label{3o} \|v_n^\varepsilon\|_{L^\infty(\Omega)} \leqslant  C_o,\qquad \qquad\mbox{ for all } \varepsilon\in (0,1),\\
\label{4o} v_n^\varepsilon  \to  2|u_n|^{q(x)-2}u_n,\mbox{ strongly in }L^r(\Omega)\mbox{ as }\varepsilon\to 0,\mbox{ for all }r\in[1,\infty).
\end{eqnarray}
Given that $v_n^\varepsilon$ belongs to $C^2(\overline{\Omega})$ and since $A_\varepsilon u$ is elliptic, Theorem 15.10 in \cite{gt} guarantees the existence of a unique solution $w_n^\varepsilon\in C^2(\overline{\Omega})$ of (\ref{Ene}).

(ii) Set $$F(z)=\int_{\Omega}\frac{|\nabla z |^{p(x)}}{p(x)}dx+\int_{\Omega}\frac{|z|^{q(x)}}{q(x)}dx-2\int_{\Omega}|u_n|^{q(x)-2}u_nzdx,$$
so that $F(z)$ is strictly convex, coercive and Fr\'echet differentiable on $\left(W_0^{1,p(x)}\cap L^{q(x)}\right)(\Omega).$
Now, if $z_n\rightharpoonup z_o$ weakly in $\left(W_0^{1,p(x)}\cap L^{q(x)}\right)(\Omega),$ then since $p\in \mathcal{P}(\Omega,\mu)$ (for definitions see \cite{dhhr}) the modulars  $\int_{\Omega}|\nabla z |^{p(x)}/p(x)dx$ and $\int_{\Omega}|z|^{q(x)}/q(x)dx$ are sequentially weakly lower semicontinuous  \cite[Thm. 3.2.9]{dhhr} and $\int_{\Omega} |u_n|^{q(x)-2}u_nzdx\in (L^{q(x)}(\Omega))^*$ we conclude $\liminf_{n\to\infty}F(z_n)\geqslant F(z_o).$
Since $F$ is bounded below, there exists $w_n\in \left(W_0^{1,p(x)}\cap L^{q(x)}\right)(\Omega)$ where $F$ attains its minimum, and since $F$ is Fr\'echet differentiable $\langle F'(w_n),\phi\rangle=0$ for all $\phi\in \left(W_0^{1,p(x)}\cap L^{q(x)}\right)(\Omega),$ i.e. $w_n$ solves $(\ref{Ene})$ in the weak sense and the uniqueness follows from the strict convexity of $F(z).$ Multiplying (\ref{Ene}) by $|w_n|^{r-2}w_n$ ($r\geqslant 2$ constant), using Young's $\varepsilon$-inequality with $\varepsilon=1/2,$ and considering that $|u_n|^{q(x)-2}u_n$ belongs to $L^\infty(\Omega)$ we obtain

\begin{align}
\label{10o}
\nonumber\int_{\Omega}|w_n|^{q(x)+r-2}dx&+(r-1)\int_{\Omega}|w_n|^{p(x)}|w_n|^{r-2}dx\\
\nonumber &\leqslant \int_{\Omega}2(n+1)^{q(x)-1}|w_n|^{r-1}dx\\
& \leqslant \frac{1}{2}\int_{\Omega}|w_n|^{q(x)+r-2}dx+2^{(q^+ +2r-3)/(q^--1)}(n+1)^{q^++r-2}|\Omega|.
\end{align}
So, by \cite[Thm. 1.3, p. 427]{fz}
\begin{equation}
\nonumber
\|w_n\|^{q^{\pm}+r-2}_{L^{q(x)+r-2}}\leqslant 2\cdot2^{(q^++2r-3)/(q^--1)}(n+1)^{q^++r-2}|\Omega|,
\end{equation}
where
\begin{eqnarray}
\nonumber
q^{\pm}=
\begin{cases}
q^+ \quad \mbox{ if }\|w_n\|_{L^{q(x)+r-2}}<1,\\
q^- \quad \mbox{ if }\|w_n\|_{L^{q(x)+r-2}}>1.
\end{cases}
\end{eqnarray}

In this way we can obtain an a priori bound for $\|w_n\|_{L^{q(x)+r-2}}$ independent of $r.$ 
Letting $r\to\infty$ we get an $L^\infty$-estimate for $w_n.$ Therefore using \cite[Thm. 1.2, p. 400]{f}
we conclude $w_n\in C^{1,\alpha}(\overline{\Omega}).$

(iii) With a similar argumentation as in (ii) we obtain
\begin{equation}\label{11o}
\|w_n^\varepsilon\|_{L^\infty(\Omega)}\leqslant C_n\quad\mbox{ for all }\quad\varepsilon>0.
\end{equation}

Multiply (\ref{Ene}) by $w_n^\epsilon,$ to obtain
$$\int_{\Omega}|w_n^\varepsilon|^{q(x)}dx+\int_{\Omega}(|\nabla w_n^\varepsilon|^2+\varepsilon)^{(p(x)-2)/2}|\nabla w_n^\varepsilon|^2dx=\int_{\Omega}v_n^\varepsilon w_n^\varepsilon dx.$$

On the other hand, note that
\begin{eqnarray}\nonumber
\int_{\Omega}|\nabla w_n^\varepsilon|^{p(x)}dx&=&\int_{\Omega}(|\nabla w_n^\varepsilon|^2)^{(p(x)-2)/2}|\nabla w_n^\varepsilon|^2dx\\
\nonumber
&\leqslant&\int_{\Omega}(|\nabla w_n^\varepsilon|^2+\varepsilon)^{(p(x)-2)/2}|\nabla w_n^\varepsilon|^2dx
\end{eqnarray}
And hence
$$\int_{\Omega}|\nabla w_n^\varepsilon|^{p(x)}dx\leqslant \int_{\Omega}v_n^\varepsilon w_n^\varepsilon dx.$$
Now use Young's inequality and the fact that $q(x),q'(x)>1$ to obtain
$$\int_{\Omega}|\nabla w_n^\varepsilon|^{p(x)}dx\leqslant \int_{\Omega}|v_n^\varepsilon|^{q'(x)}dx+\int_{\Omega}|w_n^\varepsilon|^{q(x)} dx.$$
so by (\ref{11o}) and given that $v_n\in C_0^\infty(\Omega)$ we deduce
\begin{equation}
\label{12o}\|\nabla w_n^\varepsilon\|_{L^{p(x)}(\Omega)}\leqslant C_n\quad\mbox{ for all }\quad\varepsilon>0.
\end{equation}
Together (\ref{11o}), (\ref{12o}), and  compactness Proposition \ref{rkt}  
and Lemma \ref{OTNOT} imply that
there exists a sequence $\{w_n^{\varepsilon_k}\}$ such that  
for $p^+<N$

 \begin{equation}\label{13o} w_n^{\varepsilon_k}\to w\;\mbox{ strongly in } L^{r}(\Omega),\text{ with }\;1\leqslant r^- <r(x)<r^+<\infty\end{equation}
 \begin{equation}\label{14o}\nabla w_n^{\varepsilon_k}\rightharpoonup \nabla w \quad\mbox{ weakly in }\quad L^{p(x)}(\Omega),\phantom{1\leqslant r<\infty}\end{equation}
 \begin{equation}\label{15o} \int_{\Omega}|w_n^{\varepsilon_k}|^{q(x)-2}w_n^{\varepsilon_k}v \to \int_{\Omega} |w|^{q(x)-2}wv \quad\mbox{as}\quad\varepsilon_k\to 0, \quad\mbox{for all}\quad v\in W_0^{p(x)}(\Omega).\end{equation}

 Weak convergence holds since $L^{p(x)}$ spaces are  uniformly convex \cite[Thm. 3.4.9]{dhhr}, and hence reflexive.

From this point we refer to \cite{kk} for all the notations and results concerning to sub\-dif\-ferentials.
Set $$\phi_\varepsilon(z):=\int_{\Omega}\frac{1}{p(x)}(|\nabla z|^2+\varepsilon)^{p(x)/2}dx$$
with $D(\phi_\varepsilon)=W_0^{1,p(x)}(\Omega)$ so that $\phi_\varepsilon$ is a convex operator according to definition in section 1.3.3 p. 24 in \cite{kk}. Noting that $\phi_\varepsilon$ is Fr\'echet differentiable and that actually
$$\phi_\varepsilon '(z)v=\langle A_\varepsilon z,v\rangle=\int_{\Omega}(|\nabla z|^2+\varepsilon)^{p(x)/2}\nabla z\cdot \nabla vdx.$$
So according to \cite{kk} section 4.2.2, $A_\varepsilon\in \partial \phi_\varepsilon$ where $\partial \phi_\varepsilon$
is the subdifferential of $\phi_\varepsilon.$ Hence $w_n^\varepsilon$ satisfies
$$\phi_\varepsilon (v)-\phi_\varepsilon (w_n^\varepsilon)\geqslant \int_{\Omega}(|\nabla w_n^\varepsilon|^2+\varepsilon)^{p(x)/2}\nabla w_n^\varepsilon\cdot\nabla (v-w_n^\varepsilon)dx,\qquad \forall v\in W^{1,p(x)}_0(\Omega).$$
Now, by (\ref{Ene})
\begin{equation}
\label{16o}
\phi_\varepsilon (v)-\phi_\varepsilon (w_n^\varepsilon)\geqslant \int_{\Omega}(-|w_n^\varepsilon|^{q(x)-2}w_n^\varepsilon+v_n^\varepsilon)\cdot(v-w_n^\varepsilon)dx.
\end{equation}
On the other hand, given strong convergence of $w^\varepsilon_n\to w_n$ as $\varepsilon \to 0$ and strong convergence of $v_n\to 2|u_n|^{q(x)-2}u_n$ in $L^1(\Omega),$ we have that $v_n^\varepsilon w_n^\varepsilon \to 2|u_n|^{q(x)-2}u_nw_n$ as $\varepsilon\to 0$
in $L^1(\Omega)$ since
\begin{eqnarray}
\label{16o'}
\nonumber
\int_{\Omega}|v_n^\varepsilon w_n^\varepsilon - 2|u_n|^{q(x)-2}u_nw_n|dx&\leqslant& \int_{\Omega} |v_n^\varepsilon||w_n^\varepsilon-w_n| dx\\
\nonumber& & +\int_{\Omega}|w_n|\left|v_n^\varepsilon-2|u_n|^{q(x)-2}u_n\right|dx\\
\nonumber &\leqslant& C_o\int_{\Omega} |w_n^\varepsilon-w_n| dx\\
& &+\int_{\Omega}|w_n|\left|v_n^\varepsilon-2|u_n|^{q(x)-2}u_n\right|dx,
\end{eqnarray}
given that (\ref{3o}) holds. 
That the last integral goes to zero as $\varepsilon \to 0$ follows after H\"older's inequality for variable exponent spaces $w_n\in L^r(\Omega),$ and (\ref{4o}).

Given that $\phi_\varepsilon(v)\to\phi_0(v) $ as  $\varepsilon\to 0$ for all $v\in W^{1,p(x)}(\Omega)$ and
\begin{equation}
\label{17o}
\liminf_{k\to\infty}\phi_{\varepsilon_k}(w_n^{\varepsilon_k})\geqslant \phi_{\varepsilon_k}(w)\geqslant \phi_0(w)
\end{equation}
since modulars are weakly lower semicontinuous \cite[Thm. 2.2.8]{dhhr}. Taking limits as $\varepsilon \to 0$ in (\ref{16o}),  and using (\ref{4o}), (\ref{13o}), (\ref{15o}) we get
\begin{eqnarray}
\nonumber
\phi_0(v)-\phi_0(w)&\geqslant& \int_{\Omega}\left(-|w|^{q(x)-2}w+2|u_n|^{q(x)-2}u_n\right)\cdot (v-w)dx,
\end{eqnarray}
for all $v\in W_0^{1,p(x)}(\Omega)$
which imply, by subdifferential's definition, that
\begin{equation}
\label{Io1}
\int_{\Omega}\mbox{div}(|\nabla w|^{p(x)-2}\nabla w)\cdot \nabla\varphi=
\int_{\Omega}(-|w|^{q(x)-2}w+2|u_n|^{q(x)-2}u_n)\cdot \varphi,
\end{equation}
for all $\varphi\in W_0^{1,p(x)}(\Omega).$ We conclude that $w=w_n,$ since the argument above does not depend on the choice of $\{\varepsilon_k\}.$

Multiply equation in (\ref{En}) by $w_n$ and equation in (\ref{Ene}) by $w_n^\varepsilon$ and integrate by parts to get $$\int_{\Omega}|\nabla w_n|^{p(x)}dx=-\int_{\Omega}|w_n|^{q(x)}dx+2\int_{\Omega}|u_n|^{q(x)-2}u_nw_ndx$$
$$\int_{\Omega}(|\nabla w^\varepsilon_n|^2 +\varepsilon)^{(p(x)-2)/2}|\nabla w_n^\varepsilon|^2 dx=-\int_{\Omega}|w_n^\varepsilon|^{q(x)}dx+\int_{\Omega} v_n^\varepsilon w_n^\varepsilon dx.$$
So that (\ref{4o}) and (\ref{13o}) imply
\begin{equation}
\label{18o}
\int_{\Omega}(|\nabla w_n^\varepsilon|^2+\varepsilon)^{(p(x)-2)/2}|\nabla w_n^\varepsilon|^2dx\to\int_{\Omega}|\nabla w_n|^{p(x)}dx\mbox{ as }\varepsilon\to 0.
\end{equation}
Take $v=w=w_n$ in (\ref{16o}) and let $\varepsilon \to 0$ in (\ref{16o}) to obtain
$$\limsup_{\varepsilon\to 0}\phi_\varepsilon (w_n^\varepsilon)\leqslant \phi_0(w_n),$$

Last inequality and (\ref{17o}) imply
\begin{equation}
\label{19o}
\int_{\Omega}(|\nabla w_n^\varepsilon|^2+\varepsilon)^{p(x)/2}dx\to \int_{\Omega}|\nabla w_n|^{p(x)}dx\mbox{ as }\varepsilon \to 0.
\end{equation}

Moreover, since (\ref{14o}) holds then
$$\liminf_\varepsilon \int_{\Omega}|\nabla w_n^\varepsilon|^{p(x)}dx\geqslant \int_{\Omega}|\nabla w_n|^{p(x)}$$
since modulars are weakly lower semicontinuous.

On the other hand, since $(|\nabla w_n^\varepsilon|^2)^{p(x)/2}\leqslant (|\nabla w_n^\varepsilon|^2+\varepsilon)^{p(x)/2}$ we have
$$\limsup_\varepsilon \int_{\Omega}|\nabla w_n^\varepsilon|^{p(x)}dx\leqslant \limsup_\varepsilon \int_{\Omega}(|\nabla w_n^\varepsilon|^2+\varepsilon)^{p(x)/2}dx\leqslant \int_{\Omega}|\nabla w_n|^{p(x)}dx$$
Therefore we conclude (\ref{6o}).

iv) We proceed first by noticing that
 \begin{equation}
 \label{21o}
 |u_n|^{q(x)-2}u_n\to |u|^{q(x)-2}u\quad\mbox{ strongly in }L^{q'(x)}(\Omega)\mbox{ as }n\to\infty,
 \end{equation}
by the uniform convexity of $L^{q'(x)}(\Omega).$
Multiply (\ref{En}) by $w_n$ and integrate by parts to obtain

\begin{eqnarray}
\label{22o}
\int_{\Omega} |w_n|^{q(x)}dx+\int_{\Omega} |\nabla w_n|^{p(x)}dx &=& 2\int_{\Omega}|u_n|^{q(x)-2}u_nw_ndx\\
& \leqslant & \nonumber 4\||u_n|^{q(x)-1}\|_{L^{q'(x)}(\Omega)}\|w_n\|_{L^{q(x)}(\Omega)},
\end{eqnarray}
by H\"{o}lder's inequality for variable exponent Sobolev spaces \cite[lemma 2.6.5]{dhhr}. Now, using \cite[Thm. 1.3]{fz} and (\ref{22o}) we get
\begin{equation}
\label{22mio}
\|w_n\|^{q^{\pm}}_{L^{q(x)}(\Omega)}+\|\nabla w_n\|^{p^\pm}_{L^{p(x)}(\Omega)}\leqslant C\|w_n\|_{L^{q(x)}(\Omega)},
\end{equation}

where
\begin{eqnarray}
\nonumber q^{\pm}& = &\begin{cases}
                  q^+\quad\mbox{ if }\quad\| w_n\|_{L^{q(x)}(\Omega)}<1\\
                  q^-\quad\mbox{ if }\quad\| w_n\|_{L^{q(x)}(\Omega)}\geqslant1,
                  \end{cases} \\\nonumber
         p^\pm & = & \begin{cases}
                p^+\quad\mbox{ if }\quad\| \nabla w_n\|_{L^{q(x)}(\Omega)}<1\\
                p^-\quad\mbox{ if }\quad\|\nabla w_n\|_{L^{q(x)}(\Omega)}\geqslant1,
               \end{cases}
\end{eqnarray}
The fact that $p^\pm,q^\pm>1$ imply that $\|w_n\|^{q^{\pm}}_{L^{q(x)}(\Omega)},\|\nabla w_n\|^{p^\pm}_{L^{p(x)}(\Omega)}\leqslant C.$
We use again Proposition \ref{rkt} and Lemma \ref{OTNOT}
 to obtain that, up to a subsequence $\{n_k\},$


\begin{eqnarray}
\label{23o}\nabla w_{n_k}\rightharpoonup \nabla w\quad\mbox{ weakly in }L^{p(x)}(\Omega)\\
\label{24o} w_{n_k}\rightharpoonup w\quad\mbox{ weakly in }L^{q(x)}(\Omega)\\
\nonumber w_{n_k}\to w\quad\mbox{ strongly in }L^{q(x)}(\Omega)\text{ for all }q \text{ such that } 1\leqslant q^-<q(x)<,q^+<\infty\\
\label{25o}\int_\Omega |w_{n_k}|^{q(x)-2}w_{n_k}\cdot vdx\to \int_\Omega |w|^{q(x)-2}w\cdot vdx\quad\mbox{ for all }v\in L^{q'(x)}(\Omega)\mbox{ as }k\to\infty.
\end{eqnarray}
Given that $w_n$ is solution of (\ref{En}) subdifferential's definition leads to
\begin{eqnarray}\nonumber
 \int_\Omega\frac{1}{p(x)}|\nabla v|^{p(x)}dx-\int_\Omega\frac{1}{p(x)}|\nabla w_n|^{p(x)}dx=\int_{\Omega}\frac{1}{p(x)}|\nabla v|^{p(x)}dx-\int_\Omega\frac{1}{p(x)}|\nabla w_n|^{p(x)}dx\\\label{26o}
  \geqslant \int_{\Omega}(-|w_n|^{q(x)-2}w_n+2|u_n|^{q(x)-2}u_n)(v-w_n)dx\phantom{----------}\\\nonumber
  \geqslant \int_\Omega|w_n|^{q(x)}dx-\int_\Omega |w_n|^{q(x)-2}w_nvdx+2\int_\Omega|u_n|^{q(x)-2}u_n(v-w_n)dx,
\end{eqnarray}
for all $v\in C_0^\infty(\Omega)$ and for $n$ such that $supp\,v\subset \Omega.$ Let $n=n_k\to\infty$ in (\ref{26o}) and recall (\ref{21o}), (\ref{23o}), (\ref{24o}) and (\ref{25o}) to obtain
\begin{eqnarray}\label{27o}
\int_\Omega\frac{1}{p(x)}|\nabla v|^{p(x)}dx-\int_\Omega\frac{1}{p(x)}|\nabla w|^{p(x)}dx \geqslant \int_\Omega (-|w|^{q(x)-2}w+2|u|^{q(x)-2}u)(v-w)dx,\quad
\end{eqnarray}
for all $v\in C_0^\infty(\Omega).$ Now put $v=w+tz$ with  $z\in C_o^\infty(\Omega)$ and let $t\to 0^+,$ $t\to 0^-$ in (\ref{27o}) and use the definition of Fr\'echet derivative  to see that $w$ satisfies
$$\int_\Omega |\nabla w|^{p(x)-2}\nabla w\cdot \nabla z+\int_\Omega |w|^{q(x)-2}w zdx=2\int_\Omega |u|^{q(x)-2}u zdx$$ for all $z\in C_o^\infty(\Omega).$ Hence
$$|w|^{q(x)-2}w-\Delta_{p(x)}w=|u|^{q(x)-2}u-\Delta_{p(x)}u$$
in the sense of distributions. That $w=u$  follows from well known inequality
$$|a-b|^p\leqslant C_p\left\{  (|a|^{p-2}a-|b|^{p-2}b)\cdot(a-b)\right\}^{s/2}(|a|^p+|b|^p)^{1-s/2}$$
which holds  for all $a,b\in \mathbb{R}^N$ where $s=p$ if $p\in (1,2)$ and $s=2$ if $p\geqslant 2,$ and $C_p>0$ does not depend on $a,b.$ Since the above argument does not depend  on the choice of subsequences, (\ref{23o}), (\ref{24o}) and (\ref{25o}) hold for $n_k=n.$

Taking into account (\ref{21o}), (\ref{22o}), (\ref{23o}) and (\ref{24o}) we get
\begin{eqnarray}
\nonumber 2\int_\Omega |u|^{q(x)}dx&=&\int _\Omega |u|^{q(x)}dx+\int_\Omega |\nabla u|^{p(x)}dx\\
\nonumber &\leqslant & \liminf_{n\to\infty}\left(\int _\Omega |w_n|^{q(x)}dx+\int_\Omega |\nabla w_n|^{p(x)}dx \right)\\
\nonumber & = & \lim_{n\to\infty}\left(\int _\Omega |w_n|^{q(x)}dx+\int_\Omega |\nabla w_n|^{p(x)}dx \right)\\
\nonumber &\leqslant & 2\int_\Omega |u|^{q(x)}dx.
\end{eqnarray}
Consequently
$$\lim_{n\to\infty}\left(\int _\Omega |w_n|^{q(x)}dx+\int_\Omega |\nabla w_n|^{p(x)}dx \right)=\int _\Omega |u|^{q(x)}dx+\int_\Omega |\nabla u|^{p(x)}dx $$
Further, notice that
\begin{eqnarray}
\nonumber \int_\Omega |u|^{q(x)}dx & \leqslant & \liminf_{n\to\infty} \int_\Omega |w_n|^{q(x)}dx\leqslant \limsup_{n\to\infty}\int_\Omega |w_n|^{q(x)}dx\\
\nonumber &=& \limsup_{n\to\infty}\left( \int_\Omega |w_n|^{q(x)}dx+\int_\Omega \frac{|\nabla w_n|^{p(x)}}{p(x)}dx-\int_\Omega \frac{|\nabla w_n|^{p(x)}}{p(x)}dx\right)\\
\nonumber &\leqslant &\limsup_{n\to\infty}\left( \int_\Omega |w_n|^{q(x)}dx+\int_\Omega \frac{|\nabla w_n|^{p(x)}}{p(x)}dx\right)-\liminf_{n\to\infty}\int_\Omega \frac{|\nabla w_n|^{p(x)}}{p(x)}dx\\
\nonumber &\leqslant& \int_\Omega |u|^{q(x)}dx.
\end{eqnarray}
Therefore $$\lim_{n\to\infty}\int_\Omega |w_n|^{q(x)}dx=\int_\Omega |u|^{p(x)}dx$$ and $$\lim_{n\to\infty} \int_\Omega |\nabla w_n|^{p(x)}dx=\int_\Omega |\nabla u|^{p(x)}dx.$$\end{proof}

In order to obtain a Pohozaev type inequality we introduce the function
\begin{equation}
\label{psformula}
\mathcal{F}(x,u,s):=\frac{|u(x)|^{q(x)}}{q(x)}+\frac{(|s|^2+\varepsilon)^{p(x)/2}}{p(x)}-v_n^\varepsilon (x)u(x)\end{equation}
where $s=(s_1,\dots,s_N),$ which will be used in the context of a Pucci-Serrin formula \cite{ps}.

\begin{theorem}[Pohozaev type inequality]
\label{l4.2}
Let $u$ be a weak solution of (\ref{E}) belonging to $\mathcal{P}$. Then $u$ satisfies

\begin{multline}
\label{poho2}
-\int_\Omega\frac{N}{q(x)}|u|^{q(x)}dx+\int_\Omega\frac{N-p(x)}{p(x)}|\nabla u|^{p(x)}dx
\\+\int_\Omega x\cdot\nabla p(x)\frac{|\nabla u|^{p(x)}}{p(x)^2}\log\left( e^{-1}|\nabla u|^{p(x)}\right)dx
\\-\int_\Omega x\cdot\nabla q(x)\frac{|u|^{q(x)}}{q(x)^2}\log\left( e^{-1}|u|^{q(x)}\right)dx
+R\leq 0,
\end{multline}
where  $$\displaystyle R=\frac{p^\dag-1}{p^+}\limsup_{n\to\infty}\limsup_{\varepsilon\to 0}\int_{\partial \Omega}\left(|\nabla w_n^\varepsilon|^2+\varepsilon\right)^{p(x)/2}(x\cdot\nu(x))dS, p^\dag=\min_{x\in \Omega}\left\{2,p(x)\right\},$$
and $w_n^\varepsilon$ is the solution of (\ref{Ene}) uniquely determined by $u.$
\end{theorem}

\begin{proof} In (\ref{psformula}) denote by  $\mathcal{F}_s(x,u,s)=(\partial _{s_1}\mathcal{F},\dots,\partial_{s_N}\mathcal{F}),$ so that
$$\partial_{s_i} \mathcal{F}(x,u,s)= (|s|^2+\varepsilon)^{p(x)/2-1}s_i.$$
hence we denote
$$\partial_{s_i} \mathcal{F}(x,u,\nabla u)= (|\nabla u|^2+\varepsilon)^{p(x)/2-1}\partial_iu.$$
and
$$\mathcal{F}_s(x,u,\nabla u)=(|\nabla u|^2+\varepsilon)^{(p(x)-2)/2}\nabla u.$$
So that
$$\mbox{div}\,\mathcal{F}(x,u,\nabla u)=-A_\varepsilon u, $$
where, we recall, $A_\varepsilon$ is defined after (\ref{Ene}). Finally, we denote
\begin{eqnarray}\nonumber
\nabla \mathcal{F}(x,u,\nabla u)&=&(\partial_{x_1}\mathcal{F},\dots,\partial_{x_N}\mathcal{F})\\
\nonumber& = &(\partial_{1}\mathcal{F},\dots,\partial_{N}\mathcal{F})
\end{eqnarray}
with
\begin{eqnarray*}
\partial_i \mathcal{F}&=&\partial_i\left(\frac{|u(x)|^{q(x)}}{q(x)}+\frac{(|s|^2+\varepsilon)^{p(x)/2}}{p(x)}-v_n^\varepsilon (x)u(x)\right)\\
&= & \frac{|u|^{q(x)}}{(q(x))^2}\big(\log |u|^{q(x)}-1\big)\partial_iq(x)+|u|^{q(x)-2} u\partial_i u\\
& &+\frac{(|\nabla u|^2+\varepsilon)^{p(x)/2}}{2(p(x))^2}\big(\log(|\nabla u|^2+\varepsilon)^{p(x)}-1\big)\partial_ip(x)\\
& &+(|\nabla u|^2+\varepsilon)^{p(x)/2-1}\partial_i(|\nabla u|^2)-\big[ (\partial_iv_n^\varepsilon)u+v_n^\varepsilon\partial_iu\big]
\end{eqnarray*}

We will make use the Pucci-Serrin formula \cite[Prop. 1, p. 683]{ps} in the form
\begin{multline}
\label{psf}
\int_{\partial \Omega}\Big[\mathcal{F}(x,0,\nabla u)-\nabla u\cdot \mathcal{F}_s(x,0,\nabla u)\Big](h\cdot \nu)dS\phantom{-----------------}\\
 =\int_\Omega \Big[\mathcal{F}(x,u,\nabla u)\,\mbox{div}\, h+h\cdot\nabla \mathcal{F}(x,u,\nabla u)
 -(h\cdot \nabla u)\,\mbox{div}\,\mathcal{F}_s(x,u,\nabla u)\\
-\mathcal{F}_s(x,u,\nabla u)\cdot\nabla (h\cdot\nabla u)
-au\,\mbox{div}\,\mathcal{F}_s(x,u,\nabla u)\\
-\nabla (au)\cdot\mathcal{F}_s(x,u,\nabla u)\Big]dx
\end{multline}
Taking $a$ constant, $h=x=(x_1,\dots,x_n),$ $u=w_n^\varepsilon$  equation (\ref{psf}) becomes
\begin{multline}
\label{4.20-4.21}
\int_{\partial \Omega}\frac{(|\nabla w_n^\varepsilon|^2 +\varepsilon)^{p(x)/2}}{p(x)}(x\cdot\nu)dS-\int_{\partial\Omega}(|\nabla w_n^\varepsilon|^2+\varepsilon)^{p(x)/2-1}|\nabla w_n^\varepsilon|^2(x\cdot\nu )dS = \\
=\int_{\Omega}N\left(\frac{|w_n^\varepsilon|^{q(x)}}{q(x)}+\frac{(|\nabla w_n^\varepsilon|^2+\varepsilon)^{p(x)/2}}{p(x)}-v_n^\varepsilon w_n^\varepsilon\right)dx
+\int_{\Omega}(x\cdot\nabla q(x))\frac{|w_n^\varepsilon|^{q(x)}}{(q(x))^2}\big(\log|w_n^\varepsilon|^{q(x)}-1\big)dx\\
\shoveleft{+\int_{\Omega}(x\cdot\nabla p(x))\frac{(|\nabla w_n^\varepsilon|^2+\varepsilon)^{p(x)/2}}{(p(x))^2}\big(\log(|\nabla w_n^\varepsilon|^2+\varepsilon)^{p(x)/2}-1\big)dx
 -\int_{\Omega}w_n^\varepsilon (x\cdot\nabla v_n^\varepsilon)dx}\\
 -\int_{\Omega}(|\nabla w_n^\varepsilon|^2+\varepsilon)^{(p(x)-2)/2}|\nabla w_n^\varepsilon|^2dx
+\int_{\Omega}a w_n^\varepsilon A_\varepsilon w_n^\varepsilon dx-\int_{\Omega}(\nabla (a w_n^\varepsilon)\cdot\nabla w_n^\varepsilon)(|\nabla w_n^\varepsilon|^2+\varepsilon)^{(p(x)-2)/2}dx.
\end{multline}
For the surface integrals in (\ref{4.20-4.21}) adding and subtracting  the integral  $\varepsilon\int_{\partial\Omega}(|\nabla w_n^\varepsilon|^2+\varepsilon)^{p(x)/2-1}(x\cdot\nu )dS$ we have
\begin{multline}
\label{ds}
\int_{\partial \Omega}\frac{(|\nabla w_n^\varepsilon|^2+\varepsilon)^{p(x)/2}}{p(x)}(x\cdot\nu)dS-\int_{\partial\Omega}(|\nabla w_n^\varepsilon|^2+\varepsilon)^{p(x)/2-1}|\nabla w_n^\varepsilon|^2(x\cdot\nu )dS=\\
=\int_{\partial\Omega}\left(\frac{1}{p(x)}-1 \right)(|\nabla w_n^\varepsilon|^2+\varepsilon)^{p(x)/2}(x\cdot\nu )dS+\varepsilon\int_{\partial\Omega}(|\nabla w_n^\varepsilon|^2+\varepsilon)^{p(x)/2-1}(x\cdot\nu )dS
\end{multline}
On the other hand, since $(x\cdot\nu(x))\geqslant 0$ for all $x\in\partial \Omega,$ then
\begin{eqnarray}
\label{4.22o}
 \varepsilon\int_{\partial\Omega}(|\nabla w_n^\varepsilon|^2+\varepsilon)^{p(x)/2-1}(x\cdot\nu )dS\leqslant
           \begin{cases}
                       \int_{\partial \Omega}\varepsilon^{p(x)/2} (x\cdot\nu(x))dS,\quad\mbox{ if }\quad 1 < p(x)\leqslant 2,\\
                       \int_{\partial\Omega}\frac{p(x)-2}{p(x)}(|\nabla w_n^\varepsilon|^2+\varepsilon)^{p(x)/2}(x\cdot\nu )dS+\\\phantom{-}+\int_{\partial\Omega}\frac{2}{p(x)}\varepsilon^{p(x)/2} (x\cdot\nu(x))dS,\quad\mbox{ if }\quad 2<p(x).
           \end{cases}
\end{eqnarray}
Now we analyze what happen with each term in (\ref{4.20-4.21}) as $\varepsilon\to 0.$ We begin with the last term and we continue the analysis going down to up into the equation:
\begin{enumerate}
\item  $-\int_{\Omega}(\nabla (a w_n^\varepsilon)\cdot\nabla w_n^\varepsilon)(|\nabla w_n^\varepsilon|^2+\varepsilon)^{(p(x)-2)/2}dx\to -a\int_{\Omega}|\nabla w_n|^{p(x)}dx$  by (\ref{18o}).
\item     $\int_{\Omega}a w_n^\varepsilon A_\varepsilon w_n^\varepsilon dx\to a\left(\int_{\Omega}2|u_n|^{q(x)-2}u_n w_ndx-\int_{\Omega}|w_n|^{q(x)}dx\right)$ by (\ref{Ene}) and (\ref{16o'}).
\item $-\int_{\Omega}(|\nabla w_n^\varepsilon|^2+\varepsilon)^{(p(x)-2)/2}|\nabla w_n^\varepsilon|^2dx\to -\int_{\Omega}|\nabla w_n|^{p(x)}dx $ by (\ref{18o}).

\item For the term $-\int_{\Omega}w_n^\varepsilon (x\cdot\nabla v_n^\varepsilon)dx$ we make the following estimations
    \begin{equation}
    \label{C}-\int_{\Omega}w_n^\varepsilon (x\cdot\nabla v_n^\varepsilon)dx=-\int_{\Omega}x\cdot\nabla (w_n^\varepsilon v_n^\varepsilon)dx+\int_{\Omega} v_n^\varepsilon x\cdot\nabla w_n^\varepsilon dx.
    \end{equation}
    Note that $\int_{\Omega} v_n^\varepsilon x\cdot\nabla w_n^\varepsilon dx\to 2\int_{\Omega} |u_n|^{q(x)-2}u_n x\cdot\nabla w_ndx$ as $\varepsilon \to 0,$ by a similar proof as in (\ref{16o'}).

    On the other hand, calculating the first term in the right hand side of (\ref{C}),
    \begin{eqnarray}\nonumber
    \label{C1}-\int_{\Omega}x\cdot\nabla (w_n^\varepsilon v_n^\varepsilon)dx & =& \int_{\Omega} v_n^\varepsilon w_n^\varepsilon\,\text{div}\,x \,dx-\int_{\partial\Omega}v_n^\varepsilon w_n^\varepsilon(x\cdot\nu)dS\\
    & = & N\int_{\Omega} v_n^\varepsilon w_n^\varepsilon dx.
    \end{eqnarray}

\item We claim that
    \begin{multline}
    \label{lognograd}
    \int_{\Omega}(x\cdot\nabla q(x))\frac{|w_n^\varepsilon|^{q(x)}}{(q(x))^2}\big(\log|w_n^\varepsilon|^{q(x)}-1\big)dx\to \phantom{OOOOOOOOOOOOOOOOOO} \\  \int_{\Omega}(x\cdot\nabla q(x))\frac{|w_n|^{q(x)}}{(q(x))^2}\big(\log|w_n|^{q(x)}-1\big)dx
    \end{multline}
    and
    \begin{multline}
    \label{loggrad}
     \int_{\Omega}(x\cdot\nabla p(x))\frac{(|\nabla w_n^\varepsilon|^2+\varepsilon)^{p(x)/2}}{(p(x))^2}\big(\log(|\nabla w_n^\varepsilon|^2+\varepsilon)^{p(x)/2}-1\big)dx\to\phantom{OOOOOOOO}\\
     \int_{\Omega}(x\cdot\nabla p(x))\frac{|\nabla w_n|^{p(x)}}{(p(x))^2}\big(\log|\nabla w_n|^{p(x)}-1\big)dx
    \end{multline}
    for $\eta>0.$

     Fix $$I_1:=\int_{\Omega}(x\cdot\nabla q(x))\frac{|w_n^\varepsilon|^{q(x)}}{(q(x))^2}\log|w_n^\varepsilon|^{q(x)}dx$$
    and $$I_2:=\int_{\Omega}(x\cdot\nabla p(x))\frac{(|\nabla w_n^\varepsilon|^2+\varepsilon)^{p(x)/2}}{(p(x))^2}\log(|\nabla w_n^\varepsilon|^2+\varepsilon)^{p(x)/2}dx.$$

      In order to prove (\ref{lognograd}) and (\ref{loggrad}), we estimate $I_{1}$ by  distinguishing the cases  $|w_n^\varepsilon|\leq
      1,$ and $|w_n^\varepsilon|> 1$. Notice that  the relations
      \begin{equation}\label{e3}
      \begin{gathered}
        \sup_{0\leq t\leq 1}t^{\eta}|\log t|<\infty
      \end{gathered}
      \end{equation}

      \begin{equation}\label{e4}
      \begin{gathered}
        \sup_{t>1}t^{-\eta}\log t<\infty
      \end{gathered}
      \end{equation}
       hold for $\eta >0$.\\
       Set ${\Omega}_{1}:=\{x \in \Omega:  |w_n^\varepsilon(x)|\leq 1  \} $ and  ${\Omega}_{2}:=\{x \in \Omega:  |w_n^\varepsilon (x)|>1  \}.$
       We can choose $k\in\mathbb{N}$ such that $p(x)-1/k\geq p^-.$ Since $w_n^\varepsilon\in L^{p^-}(\Omega)$ and in ${\Omega}_1,$ $|w_n^\varepsilon(x)|\leq 1$
       we have
      \begin{equation}
      \label{Leb1}
       \left|(x\cdot\nabla q(x))\frac{|w_n^\varepsilon|^{q(x)}}{(q(x))^2}\log|w_n^\varepsilon|^{q(x)}\right|\leq C|w_n^\varepsilon(x)|^{p(x)-1/m}\leq C|w_n^\varepsilon(x)|^{p^-},
       \end{equation}
        for $m>k.$

    For ${\Omega}_2$ we can choose $k'$ such that $p(x)+1/k'\leq (p(x))^*=Np(x)/(N-p(x)).$
   So
    \begin{equation}
    \label{Leb2}
     \left|(x\cdot\nabla q(x))\frac{|w_n^\varepsilon|^{q(x)}}{(q(x))^2}\log|w_n^\varepsilon|^{q(x)}\right|\leq C|w_n^\varepsilon(x)|^{p(x)+1/m}\leq C|w_n^\varepsilon(x)|^{(p(x))^*},
     \end{equation}
     for $m>k',$ and $x\in{\Omega}_2.$
     Therefore (\ref{Leb1}), (\ref{Leb2}), and the convergence of $w_n^\varepsilon$ in Lemma \ref{3.1} imply that there exists $h(x)\in L^1(\Omega)$ such that
    \begin{equation}
     \label{acot}
      \left|(x\cdot\nabla q(x))\frac{|w_n^\varepsilon|^{q(x)}}{(q(x))^2}\log|w_n^\varepsilon|^{q(x)}\right|\leq h(x)
      \end{equation}
       On the other hand, given the convergence Lemma \ref{3.1}, assertion (\ref{7o}) and the continuity of the log function, we may conclude that
    \begin{equation}
    \label{ult}
     (x\cdot\nabla q(x))\frac{|w_n^\varepsilon|^{q(x)}}{(q(x))^2}\log|w_n^\varepsilon|^{q(x)}\to(x\cdot\nabla q(x))\frac{|w_n|^{q(x)}}{(q(x))^2}\log|w_n|^{q(x)}
     \end{equation}
      $a.e.$ in $\Omega$ as $\varepsilon \to 0.$ With (\ref{acot}), (\ref{ult}), and the Lebesgue convergence Theorem the claim of point (5) follows.

\item Finally, $\displaystyle{\int_{\Omega}N\left(\frac{|w_n^\varepsilon|^{q(x)}}{q(x)}+\frac{(|\nabla w_n^\varepsilon|^2+\varepsilon)^{p(x)/2}}{p(x)}\right)dx\to \int_{\Omega} N\left(\frac{|w_n|^{q(x)}}{q(x)}+\frac{|\nabla w_n|^{p(x)}}{p(x)}\right)dx}$ as $\varepsilon\to 0$ by (\ref{18o}) and (\ref{7o}).
\end{enumerate}
Considering points (1) to (6), identities (\ref{4.20-4.21}), (\ref{ds}), and inequality (\ref{4.22o}), we obtain
\begin{multline}
\label{inl4.2}
N\int_\Omega\frac{|w_n|^{q(x)}}{q(x)}dx+\int_\Omega\frac{N-p(x)}{p(x)}|\nabla w_n|^{p(x)}dx
+\int_\Omega x\cdot\nabla p(x)\frac{|\nabla w_n|^{p(x)}}{p(x)^2}\left(\log |\nabla w_n|^{p(x)}-1\right)dx\\+\int_\Omega x\cdot\nabla q(x)\frac{|w_n|^{q(x)}}{q(x)^2}\left(\log |w_n|^{q(x)}-1\right)dx
+2\int_\Omega|u_n|^{q(x)-2}u_nx\cdot \nabla w_ndx\\
+a\left(\int_\Omega 2|u_n|^{q(x)-2}u_n w_ndx-\int_\Omega|w_n|^{q(x)}dx-\int_\Omega |\nabla w_n|^{p(x)}dx\right)
+R_n\leq 0,
\end{multline}
where $R_n=\frac{p^\dag-1}{p^+}\limsup_{\varepsilon\to 0}\int_{\partial \Omega}\left(|\nabla w_n^\varepsilon|^2+\varepsilon\right)^{p(x)/2}(x\cdot\nu(x))dS,$ and $p^\dag=\min_{x\in \Omega}\left\{2,p(x)\right\}.$

Now let $n\to\infty$ in (\ref{inl4.2}) and take into account (\ref{8o}), (\ref{9o}) to obtain

\begin{multline}
\label{poho1}
N\int_\Omega\frac{|u|^{q(x)}}{q(x)}dx+\int_\Omega\frac{N-p(x)}{p(x)}|\nabla u|^{p(x)}dx
+\int_\Omega x\cdot\nabla p(x)\frac{|\nabla u|^{p(x)}}{p(x)^2}\left(\log |\nabla u|^{p(x)}-1\right)dx\\+\int_\Omega x\cdot\nabla q(x)\frac{|u|^{q(x)}}{q(x)^2}\left(\log |u|^{q(x)}-1\right)dx
+2\int_\Omega|u|^{q(x)-2}u \left(x\cdot \nabla u \right)dx\\
+a\left(\int_\Omega |u|^{q(x)}dx-\int_\Omega |\nabla u|^{p(x)}dx\right)
+R\leq 0,
\end{multline}
where $\displaystyle R=\frac{p^\dag-1}{p^+}\limsup_{n\to\infty}\limsup_{\varepsilon\to 0}\int_{\partial \Omega}\left(|\nabla w_n^\varepsilon|^2+\varepsilon\right)^{p(x)/2}(x\cdot\nu(x))dS.$

Further, notice that since $u$ is a weak solution of (\ref{E}),
\begin{equation}
\label{a0}
\int_\Omega |u|^{q(x)}dx-\int_\Omega |\nabla u|^{p(x)}dx=0.
\end{equation}
  In fact, multiplying (\ref{E}) by $\varphi\in W^{1,p(\cdot)}_0(\Omega),$ and integrating by parts, we have
$$\int_\Omega |\nabla u|^{p(x)-2}\nabla u dx=\int_\Omega |u|^{q(x)-2}u\varphi dx.$$ Taking $\varphi=u$ we get (\ref{a0}) as wanted.
 On the other hand,
 \begin{eqnarray}
    \label{C2}
    \nonumber\int_{\Omega} \frac{x\cdot \nabla |u|^{q(x)}}{q(x)}dx&=&\int_{\Omega} |u|^{q(x)-2}u(x\cdot  \nabla u)dx\\
    & &+\int_{\Omega}\frac{1}{q(x)^2}|u|^{q(x)}\log |u|^{q(x)}(x\cdot\nabla q(x))dx,
    \end{eqnarray}
    so that
     \begin{eqnarray}
     \label{C3}\nonumber
     \int_{\Omega} \frac{x\cdot \nabla |u|^{q(x)}}{q(x)}dx&=&-\int_{\Omega} \text{div}\left(\frac{x}{q(x)}\right)|u|^{q(x)}dx+\int_{\partial\Omega}|u|^{q(x)}\frac{\partial}{\partial \nu}\left(\frac{x}{q(x)}\right)dS\\
    & & -N\int_{\Omega} \frac{|u|^{q(x)}}{q(x)}dx+\int_{\Omega}\frac{|u|^{q(x)}x\cdot\nabla q(x)}{q(x)^2}dx.
    \end{eqnarray}
    Hence from (\ref{C2}), (\ref{C3})

    \begin{eqnarray}
    \label{C4} \nonumber
    \int_{\Omega} |u|^{q(x)-2}u(x\cdot \nabla u)dx& = &-N\int_{\Omega} \frac{|u|^{q(x)}}{q(x)}dx\\
    & & +\int_{\Omega} \frac{|u|^{q(x)}x\cdot\nabla q(x)}{q(x)^2}\left(1-\log |u|^{q(x)} \right)dx
    \end{eqnarray}
    We derive inequality (\ref{poho2}) by substituting  (\ref{a0}) and (\ref{C4}) in (\ref{poho1}) .
    \end{proof}


\section{Nonexistence of Nontrivial Solutions}
\label{NES}

Now we can state a Non Existence Theorem which is a generalization to variable exponent Sobolev spaces of Theorem III, p. 142 in \cite{o1}. The proofs are similar to those in \cite{o1}, but are included here for the reader's convenience.

\begin{theorem}\label{pti}
Consider the Problem (\ref{E}), where $\Omega\subset \mathbb{R}^N$ is a bounded domain  of Class $C^1,$  $p(\cdot)$ is a log-H\"older exponent with $1<p^-\leqslant p(x)\leqslant p^+<N.$ Let $\mathcal{P}$ be as defined in (\ref{P}). Then we have:
\begin{enumerate}
\item[i)] If $\Omega$ is star-shaped and $q^->(p^+)^*$ then Problem (\ref{E}) has not  a nontrivial weak solution belonging  to $\mathcal{P}\cap\mathcal{E}$ where
    $$\mathcal{E}=\left\{u:\int_\Omega \log \left(\frac{\left(|\nabla u|^{p(x)}e^{-1}\right)^{\frac{x\cdot\nabla p}{p^2}|\nabla u|^{p(x)}}}{\left(|u|^{q(x)}e^{-1}\right)^{\frac{x\cdot\nabla q}{q^2}|u|^{q(x)} }} \right)dx\geqslant 0  \right\}.$$

\item[ii)] If $\Omega$ is strictly star-shaped and $q^-=(p^+)^*$ then Problem (\ref{E}) has not  a nontrivial weak solution of definite sign belonging  to $\mathcal{P}\cap\mathcal{E}.$
\end{enumerate}
\end{theorem}
\begin{proof}
i) If $\Omega$ is star-shaped, $R\geqslant 0$ in (\ref{poho2}). Then it follows that
$$\left(\frac{N- p^+}{p^+}-\frac{N}{q^-}\right)\int_\Omega |u|^{q(x)}dx\leqslant 0.$$
So $u\equiv 0.$

ii) If $\Omega$ is strictly star-shaped, $R=0$ in (\ref{poho2}), so
$$0=R\geqslant \rho\limsup_{n\to\infty}\limsup_{\varepsilon\to 0}\int_{\partial\Omega}\left(|\nabla w_n^\varepsilon|^2+\varepsilon\right)^{p(x) /2}dS .$$
Since $\rho>0$ we have
$$0=\limsup_{n\to\infty}\limsup_{\varepsilon\to 0}\int_{\partial\Omega}\left(|\nabla w_n^\varepsilon|^2+\varepsilon\right)^{p(x) /2}dS .$$
 Multiplying the PDE in (\ref{Ene}) by $v(x)\equiv 1,$ integrating by parts, and taking $\limsup$ as $\varepsilon\to 0$ and $n\to \infty$ we obtain
$$\left|\int_\Omega |u|^{q(x)-2}udx\right|\leqslant C\limsup_{n\to\infty}\limsup_{\varepsilon\to 0}\int_{\partial\Omega}\left(|\nabla w_n^\varepsilon|^2+\varepsilon\right)^{p(x) /2}dS=0, \quad C\geqslant 0.$$
Therefore $\int_\Omega |u|^{q(x)-2}udx=0.$ 
\end{proof}

\end{document}